\documentclass
[11pt,a4paper,oneside]{article}

\usepackage[english]{babel}

\usepackage{a4wide}
\usepackage{amsmath}
\usepackage{amsfonts}
\usepackage{amssymb}
\usepackage{amsthm}
\usepackage{makeidx}
\usepackage{graphicx}
\usepackage{epstopdf}
\usepackage{fancyhdr}
\usepackage{showidx}
\usepackage{mathrsfs}
\usepackage{subfigure}

\providecommand{\abs}[1]{\left| #1 \right|}
\providecommand{\norm}[1]{\left\| #1 \right\|}
\providecommand{\tonde}[1]{\left( #1 \right)}
\providecommand{\quadre}[1]{\left[ #1 \right]}

\newcommand{\R}{\mathbb{R}}

\newcommand{\N}{\mathbb{N}}

{\left\lbrace\begin{array}{@{}l@{}}}%
{\end{array}\right.}
\numberwithin{equation}{section}
\newtheorem{thm}{Theorem}[section]
\newtheorem{cor}[thm]{Corollary}

\newtheorem{prop}[thm]{Proposition}
\theoremstyle{definition}
\theoremstyle{definition}
\theoremstyle{definition}

\theoremstyle{remark}
\newtheorem{rem}[thm]{Remark}

\def\1{\mathbb I}

\def\l{\lambda}
\def\m{\mu}
 \def\r{\rho}

\def\Ent{\hbox{Ent}}
\def\half{\frac{1}{2}}

\title{Hypercontractivity of  a  semi-Lagrangian scheme for  Hamilton-Jacobi equations}
\author{F. Camilli\thanks{camilli@sbai.uniroma1.it}, P. Loreti\thanks{loreti@sbai.uniroma1.it}, C. Pocci\thanks{pocci@sbai.uniroma1.it}\\
SBAI, Dip. di Scienze di Base e Applicate per l'Ingegneria,\\
"Sapienza", Universit{\`a} Roma
 }

\begin{document}
\maketitle
\begin{abstract}
The equivalence between logarithmic Sobolev inequalities and hypercontractivity  of solutions of Hamilton-Jacobi equations
has been proved in \cite{BGL}. We consider a semi-Lagrangian approximation scheme for the Hamilton-Jacobi equation
and we prove that the solution of the discrete problem satisfies a   hypercontractivity  estimate. We apply this property to obtain
an error estimate of the set where the truncation error is concentrated.
\end{abstract}
\begin{description}
\item [{\bf MSC 2000}:]   49M25, 65N15.
\item [{\bf Keywords}:] Hamilton-Jacobi equation, semi-Lagrangian scheme, hypercontractivity, error estimate.
\end{description}
\section{Introduction}
Consider the Hamilton-Jacobi equations
\begin{equation}\label{HJ}
   u_t+H(Du)=0,\qquad x\in\R^N,\,t>0
\end{equation}
where $H:\R^N\to \R$ is   convex and $\lim_{\abs p\to +\infty} \frac{H(p)}{p}=+\infty$. For $f$   a Lipschitz continuous function, define
\begin{equation}\label{HL}
    S_tf(x)=\inf_{y\in\R^N}\left\{f(y)+tL\left(\frac{x-y}{t}\right)\right\}
\end{equation}
where $L$ is the Legendre transform of $H$, i.e.
\begin{equation}\label{Lag}
 L(q)=\sup_{q\in\R^N}\{p\cdot q-H(p)\}.
\end{equation}
The family $(S_t)_{t\ge 0}$ defines a semigroup with infinitesimal generator $-H(Df)$ and
the solution of the equation \eqref{HJ} with initial datum $u(x,0)=f(x)$ is given by $u(x,t)=S_tf(x)$.\\
The link between Hamilton-Jacobi equations with $H(p)=\frac{\abs{p}^2}{2}$ and logarithmic Sobolev inequality (LSI in short) is given in \cite{BGL}. We recall that the classical LSI can  be written as
\begin{equation}\label{LSI}
\mathcal E_{\mu}(u^2)\leq \frac{2}{\rho}\int_{\R^N}\abs{Du}^2 d\mu,
\end{equation}
where
\begin{equation*}\label{entropia}
\mathcal E_{\mu}(u^2)=\int_{\R^n} u^2 \log u^2 d\mu-\int_{\R^n} u^2 d\mu \log \int_{\R^n} u^2
d\mu,
\end{equation*}
 $\mu$ is a probability measure and $\rho$ is a positive real number. The typical example of measure  satisfying the inequality \eqref{LSI} is the canonical Gaussian measure with density $(2\pi)^{-N/2} e^{-|x|^2/2}$ with respect to the Lebesgue measure in $\R^N$ (in this case $\rho=1$). \\
In \cite{BGL} the authors prove that if $\mu$ satisfies  \eqref{LSI}, then for every $t>0$ and $a\in\R$ the solution of \eqref{HJ} satisfies the hypercontractivity estimate
\begin{equation}\label{Hyp1}
    \|e^{S_t f}\|_{a+\rho t}\le \|e^f\|_a
\end{equation}
($\norm{\cdot}_p$ is the $L^p$-norm associated to the   measure  $\mu$).
Conversely, if \eqref{Hyp1} holds for all $t>0$ and some $a\neq 0$ then \eqref{LSI} holds. \\
Aim of this paper is to show that a similar hypercontractivity property holds for a semi-Lagrangian approximation of      \eqref{HJ}. Semi-Lagrangian schemes
are a well studied class of approximation schemes for Hamilton-Jacobi equation (see \cite{BCD}, \cite{FF}). For a fixed discretization step $h$, the semi-Lagrangian scheme generates a discrete-time semigroup $(Q_n)_{n\in \N}$  and the  solution of the approximate      equation
 with initial datum $u(x,0)=f(x)$ is given by $u(x,nh)=(Q_n f)(x)$.
 We show that  if $\mu$ satisfies \eqref{LSI}, then the  semigroup $(Q_n)_{n\in \N}$ satisfies at the discrete  time $nh$  the
 hypercontractivity estimate
 \begin{equation}\label{Hyp2}
    \|e ^{Q_n f}\|_{\l_n}\le  \|e ^{f}\|_{a}\prod_{k=1}^{n} (1+C\l_{k-1} h )^{\frac{1}{\l_{k}}},
    \end{equation}
where $\l_n=a+\r n h$.
And, as in the continuous case, also the converse is true.\\
 It is by now classical  that approximation schemes for Hamilton-Jacobi equations give  an error estimate of order $h^{1/2}$ in the $L^\infty$-norm, while $L^p$-estimates are known only in some particular cases (see \cite{bfz}, \cite{le}, \cite{lt}).  We apply  \eqref{Hyp2} to give an estimate of the set where the truncation error is concentrated showing that its measure decays exponentially. We note that similar estimates are obtained for the approximation of stochastic differential equations via Euler schemes (see \cite{LM}, \cite{MT}).\par
 The paper is organized as follows.\\
 In Section \ref{Section1} we study the property of the discrete-time  semigroup. In Section \ref{Section2} we prove the equivalence between the hypercontractivity estimate and the logarithmic Sobolev inequality with respect to a gaussian measure, while in Section \ref{Section3} we study a similar property for the Lebesgue measure. Finally, in  Section \ref{Section4}, we prove the concentration estimate.


\section{The discrete semigroup}\label{Section1}
Fixed a discretization step $h>0$ and given a continuous function $u_0$,  consider the   semi-Lagrangian scheme for
 \eqref{HJ}  (see \cite{FF})
\begin{equation}\label{HJh}
    \frac{u_{n+1}(x)-u_n(x)}{h}+\sup_{q\in\R^N}\{-\frac{u_n(x-hq)-u_n(x)}{h}-L(q)\}=0,
\end{equation}
or equivalently
\begin{equation}\label{HJhb}
     u_{n+1}(x)=\inf_{q\in\R^N}\{ u_n(x-hq) +hL(q)\}.
\end{equation}
Given a continuous function $f$,  define
\begin{equation}\label{SGh}
 ( Q_nf) (x)=\inf_{q_0,\dots,q_{n-1}\in \R^N}\left\{f\Big(x-\sum_{k=0}^{n-1}hq_k\Big)+\sum_{k=0}^{n-1}
    hL(q_k)\right\}
\end{equation}
(with the convention that $\sum_{k=0}^{-1}=0$).
In the next proposition we show that the family $(Q_n \cdot)_{n\in \N}$ generates a discrete-time semigroup
giving the solution of \eqref{HJh} with initial datum $u_0(x)=f(x)$.
\begin{prop}\label{SG}\hfill
\begin{itemize}
  \item [1)]$(Q_{n+m}f)(x)=(Q_n(Q_mf))(x)$ for any $n,m \in \N$
and $(Q_0f)(x)=f(x)$. Moreover
$(Q_n(f+c))(x)=(Q_n f)(x)+c$ for any constant $c\in \R$.\vskip 6pt
\item [2)] $u_n(x)=(Q_n f)(x)$ is the solution of \eqref{HJh} with $u_0(x)=f(x)$.
\item[3)] If $f$ is bounded and Lipschitz continuous,  then $Q_n f$ is bounded and Lipschitz continuous and
\begin{align}
  & \|Q_{n+1}f\|_\infty\le \|Q_n f\|_\infty+ h\max\{-H(0),L(0)\}\quad \text{for any $n\in\N$}\label{P1} \\
  &\left|(Q_nf)(x)-(Q_nf)(y)\right|\le \|D f\|_\infty |x-y|\quad \text{for any $n\in\N$, $x,y\in\R^N$}\label{P2}  \\
  &\left|(Q_m f)(x)-(Q_nf)(x)\right|\le C |n-m|h\quad\text{for any $n\in\N$, $x\in\R^N$.}\label{P3}
\end{align}
with $C$ independent of $h$.
\item [4)] If $f$ is semiconcave, then $Q_n f$ is semiconcave in $x$, uniformly in $n$.
\end{itemize}
\end{prop}
\begin{proof}
To prove 1),  observe that
\begin{align*}
    &(Q_n(Q_mf))(x)=\inf_{q_1,\dots,q_n}\left\{(Q_m f)\Big(x-\sum_{k=0}^{n-1}hq_k\Big)+\sum_{k=0}^{n-1}
    hL(q_k)\right\}\\
    &=\inf_{q_1,\dots,q_n}\left\{\inf_{q_{n+1},\dots,q_m}\left\{f\Big(x-\sum_{k=0}^{n-1}hq_k-\sum_{k=0}^{m-1}hq_{n+1+k}\Big)+\sum_{k=0}^{m-1}
    hL(q_{n+1+k})\right\}+\sum_{k=0}^{n-1} hL(q_k)\right\}\\
    &=(Q_{n+m}f)(x).
\end{align*}
The commutativity with the constants is immediate.\\
Set  $u_n=Q_nf$. By  1), $u_{n+1}(x)=(Q_{n+1}f)(x)=(Q_1(Q_nf))(x)=(Q_1 u_n)(x)$, hence $u_n$ is the  solution of \eqref{HJhb}  with $u_0(x)=(Q_0f)(x)=f(x)$.\\
By \eqref{HJhb} for $q=0$, we have
\begin{equation}\label{e1}
u_{n+1}(x)\le u_n(x)+hL(0).
\end{equation}
Moreover since $L(q)\ge -H(0)$ for any $q\in\R^N$, we have
\begin{equation}\label{e2}
u_{n+1}(x)\ge -\|u_n\|_\infty-hH(0).
\end{equation}
By \eqref{e1} and \eqref{e2}, we get \eqref{P1} for $u_n=Q_nf$.\\
Since $u_n=Q_nf$ is continuous and $L$ is superlinear, there exists $R_n$ (increasing with respect to $n$ but upper bounded uniformly in $n$ and $h$) such that, defined $M_n(x)=\arg\inf\{u_n(x-hq)+hL(q)\}$, then $M_n(x)\subset B(0,R_n)$ and the infimum in \eqref{HJhb} is obtained. Given $x,y\in\R^N$ and $q^*\in M_n(x)$, by \eqref{HJhb}   we get
\begin{align*}
    u_{n+1}(x)-u_{n+1}(y)\le u_n(x-hq^*)-u_n(y-hq^*)\le \|Du_n\|_\infty|x-y|=\|D(Q_nf)\|_\infty|x-y|.
\end{align*}
Iterating in the previous inequality, we get \eqref{P2}. \\
For $q^*\in M_n(x)$ in \eqref{HJhb}, we have
\[u_{n+1}(x)-u_n(x)=u_n(x-hq^*)-u_n(x)+hL(q^*)\le \|Du_n\|_\infty h|q^*|+hL(q^*)\]
which gives \eqref{P3} since $M_n(x)$ is uniformly bounded. \\
%
%
Assume that $u_n$ is  semi-concave  with constant  $C_n$. If $q^* \in M_n(x)$, then
\begin{align*}
     u_{n+1}(x)&=u_n(x-hq^*)+hL(q^*)\\
     u_{n+1}(x\pm y)&\le u_n(x-hq^*\pm y)+hL(q^*).
\end{align*}
Hence
\begin{align*}
 u_{n+1}(x+y )&+u_{n+1}(x-y )-2u_{n+1}(x) \\
 &\le u_n(x-hq^*+ y)+u_n(x-hq^*- y)-2u_n(x-hq^*)\le   C_n|y|^2.
\end{align*}
Hence $C_{n+1}\le C_n $ and iterating
\[C_{n+1}\le C_0\]
where $C_0$ is the  semi-concavity constant of $u_0=f$.
\end{proof}
\begin{rem}
Note that \eqref{HJhb}
can be rewritten as
\[(Q_{n+1}f)(x)=\inf_{q\in\R^N}\left\{(Q_nf)(x)+hL\left(\frac{x-y}{h}\right)\right\}.\]
Hence the discrete semigroup $Q_n$  is obtained by  considering  the continuous
semigroup $S_{t}$ only at the discrete time $t=nh$, $n\in\N$.
\end{rem}

\section{Hypercontractivity of the discrete semigroup with respect to  Gaussian measures}\label{Section2}
We prove the hypercontractivity of the discrete semigroup $(Q_n \cdot)_{n\in \N}$ with respect to a measure $\mu$ satisfying \eqref{LSI}.  In this section $\norm{\cdot}_p$ is the $L^p$-norm associated to the   measure  $\mu$. We  only consider  the case  $H(p)=\half |p|^2$ and therefore $L(q)=\half |q|^2$, but the results can be extended to more general Hamiltonians as in \cite{BGL}.
\begin{thm}\label{stimeF}
    If $\m$ is absolutely continuous with respect to the Lebesgue measure and satisfies the  logarithmic Sobolev inequality \eqref{LSI}, then
    for any $f$ Lipschitz continuous and   for any $n$, any $h>0$,  $a \in\R$
    \begin{equation}\label{hyperconc2}
    \|e ^{Q_n f}\|_{\l_n}\le  \|e ^{f}\|_{a}\prod_{k=1}^{n} (1+C\l_{k-1} h )^{\frac{1}{\l_{k}}}.
    \end{equation}
    where $\l_n=a+\r n h$.\par
Conversely if \eqref{hyperconc2} holds for any smooth function $f$, for any $n$, any $h>0$ and some $a\neq 0$, then the measure $\m$
satisfies the logarithmic Sobolev inequality \eqref{LSI}.
\end{thm}
\begin{proof}
We define
\begin{equation}\label{F}
    F_n=\|e ^{Q_n f}\|_{\l_n}=\left(\int e^{\l_n Q_nf(x)}d\m\right)^{\frac{1}{\l_n}}
\end{equation}
and we prove that
 \begin{equation}\label{hyperconc1}
    F_{n+1}\le F_n(1+C\l_n h )^{\frac{1}{\l_{n+1}}}.
    \end{equation}
Then \eqref{hyperconc2} is obtained iterating over $n$.\\
We consider a measure $\m$ which satisfies \eqref{LSI} and
we start with the identity
\begin{equation}\label{h1}
   F_{n+1}^{\l_{n+1}}-F_n^{\l_{n}}=\int e^{{\l_{n+1}Q_{n+1}f}}d\m-\int e^{{\l_{n}Q_{n}f}}d\m.
\end{equation}
We consider first the term on the right hand side of \eqref{h1}. We have
\begin{equation}\label{h2}
\begin{split}
   & F_{n+1}^{\l_{n+1}}-F_n^{\l_{n}}=F_{n+1}^{\l_{n+1}}-F_{n}^{\l_{n+1}}+F_n^{\l_{n+1}}-F_n^{\l_{n}}=\\
   & e^{\l_{n+1}\ln(F_n)}\left[ e^{\l_{n+1}(\ln(F_{n+1})-\ln(F_n))}-1\right]+e^{\l_n\ln(F_n)}
   \left[e^{(\l_{n+1}-\l_n)\ln(F_n)}-1\right]=\\
   & F_n^{\l_{n+1}}\left[\left(\frac{F_{n+1}}{F_n}\right)^{\l_{n+1}}-1\right]+F_n^{\l_n}[e^{(\l_{n+1}-\l_n)\ln(F_n)}-1].
\end{split}
\end{equation}

We now consider the term on the right hand side of \eqref{h1}
\begin{equation}\label{h3}
\begin{split}
    &\int e^{{\l_{n+1}Q_{n+1}f}}d\m-\int e^{{\l_{n}Q_{n}f}}d\m=\int e^{{\l_{n+1}Q_{n+1}f}}-e^{{\l_{n+1}Q_{n}f}}d\m+\\
    &\int e^{{\l_{n+1}Q_{n}f}}-e^{{\l_{n}Q_{n}f}}d\m=\int e^{{\l_{n+1}Q_{n}f}}
    \left[e^{{\l_{n+1}(Q_{n+1}f-Q_nf)}}-1\right]d\m+\\
    &\int e^{{\l_{n}Q_{n}f}} \left[e^{{(\l_{n+1}-\l_n)Q_nf}}-1\right]d\m.
\end{split}
\end{equation}
By \eqref{h2} and \eqref{h3}, multiplying for $\l_n$, we get
\begin{equation}\label{h4}
    \begin{split}
    & \l_nF_n^{\l_{n+1}}\left[\left(\frac{F_{n+1}}{F_n}\right)^{\l_{n+1}}-1\right]=\\
    & \l_n\int e^{{\l_{n+1}Q_{n}f}}\left[e^{{\l_{n+1}(Q_{n+1}f-Q_nf)}}-1\right]d\m+
    \l_n\int e^{{\l_{n}Q_{n}f}} \left[e^{{h\r Q_nf}}-1\right]d\m-\\
   & \l_n F_n^{\l_n}[e^{h\r\ln(F_n)}-1]=h\r\Ent(e^{\l_n Q_nf})+\l_n\int e^{{\l_{n+1}Q_{n}f}}
    \left[e^{{\l_{n+1}(Q_{n+1}f-Q_nf)}}-1\right]d\m+\\
   &\l_n \int e^{{\l_{n}Q_{n}f}} \left[e^{{h\r Q_nf}}-1-h\r Q_nf\right]d\m
    -\l_nF_n^{\l_n}[e^{h\r\ln(F_n)}-1-h\r\ln(F_n)]\le\\
    &h\l_n^2\int e^{\l_n Q_nf}\frac{|DQ_n f|^2}{2}d\m+\l_n\int e^{{\l_{n+1}Q_{n}f}}
    \left[e^{{\l_{n+1}(Q_{n+1}f-Q_nf)}}-1\right]d\m+\\
   &\l_n \int e^{{\l_{n}Q_{n}f}} \left[e^{{h\r Q_nf}}-1-h\r Q_nf\right]d\m
    -\l_nF_n^{\l_n}[e^{h\r\ln(F_n)}-1-h\r\ln(F_n)].
 \end{split}
\end{equation}
Observing that $\l_n\le \l_{n+1}$ , $|e^{{h\r Q_nf}}-1-h\r Q_nf|\le Ch^2$, $Q_{n+1}f\le Q_n f$ and  the last term  on the right hand side
of \eqref{h4} is negative  we get
\begin{equation*}
    \begin{split}
    & \l_nF_n^{\l_{n+1}}\left[\left(\frac{F_{n+1}}{F_n}\right)^{\l_{n+1}}-1\right]\le
    C \l_n^2h F_n^{\l_n}+C\l_nF_n^{\l_n}h^2\le C\l_n^2 h F_n^{\l_{n+1}}
%
%
 \end{split}
\end{equation*}
and therefore
\begin{equation}\label{h6}
   F_{n+1}^{\l_{n+1}}\le F_n^{\l_n}(1+C\l_n h )\le F_n^{\l_{n+1}}(1+C\l_nh)
\end{equation}
which gives \eqref{hyperconc1}.\par
To prove the converse, given $t>0$, let $t_n=nh$ converging to $t$ for $h\to 0$ and $n\to \infty$.
By standard stability results in viscosity solution theory (see \cite{BCD}) $Q_nf$ converges to $S_t f$ uniformly in $x$, where
$S_t$ is the semigroup associated to \eqref{HJ}.
Moreover $\l_n\to\l(t)=a+\r t$ and  $(1+C\l_{k-1} h )^{1/\l_{k}}\to 1$.
Hence by \eqref{hyperconc2}, for $h\to 0$ we get the hypercontractivity of the continuous semigroup
\begin{equation}\label{hypercont}
    \|e^{S_tf(x)}\|_{\l(t)}\le \|e^f\|_a.
\end{equation}
Then the statement follows since it is well known that if the inequality \eqref{hypercont} holds for some $a\neq 0$
and for any smooth function $f$, the measure $\m$ satisfies \eqref{LSI} (see \cite{BGL}).
\end{proof}
\begin{figure}[h!]
\begin{center}
\includegraphics[scale=0.8]{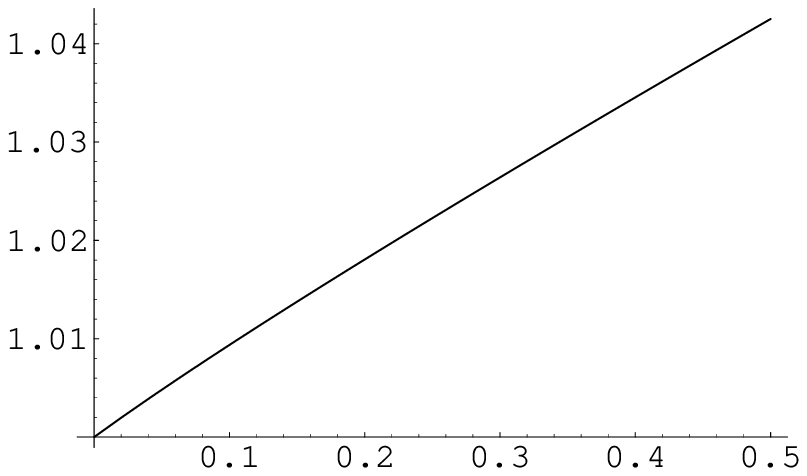} \quad \quad
\end{center}
\caption{Behavior of $f(h)=\prod_{k=1}^{10} (1+C\l_{k-1}h)^{\frac{1}{\l_{k}}}$, fixed $\rho=1$, $C=0.01$.}
\label{fig1}
\end{figure}
\section{Hypercontractivity  of the discrete semigroup with respect to the  Lebesgue measure}\label{Section3}
In \cite{gentil03}, it is proved that the semigroup $S_t$ satisfies the following   optimal hyper-contractivity inequality 
with respect to the Lebesgue measure  
\begin{equation}\label{ineq_gentil}
\norm{e^{S_t f} }_{\beta}\leq \tonde{\frac{\alpha}{\beta}}^{\frac{N}{\alpha \beta}\frac{\alpha+\beta}{2}}\norm{e^f}_{\alpha}\tonde{\frac{\beta-\alpha}{t}}^{\frac{N}{2}\frac{\beta-\alpha}{\alpha \beta}}\quadre{\int e^{-L(x)}dx}^{-\frac{\beta-\alpha}{\alpha \beta}} 
\end{equation}
where $\norm{\cdot}_p$ is the $L^p$-norm associated to the Lebesgue measure on $\R^N$.
In thi section we study  hypercontractivity of the discrete semigroup $Q_n$ with respect to  the Lebesgue measure. We assume for simplicity that $H(p)=\half |p|^2$ and therefore $L(q)=\half |q|^2$.\par 
\begin{thm}\label{Leb}
For any smooth function $f \in \R^N$, $ 0<\alpha\leq\beta$, we have
\begin{equation}\label{ineq_PL}
\norm{e^{Q_n f}}_{\beta}\leq\big( {\frac{\alpha}{\beta}}\big)^{\frac{N}{\alpha \beta}\frac{\alpha+\beta}{2}}\norm{e^f}_{\alpha}\tonde{\frac{\beta-\alpha}{nh}}^{\frac{N}{2}\frac{\beta-\alpha}{\alpha \beta}}\tonde{2\pi}^{-\frac{N}{2}\frac{\beta-\alpha}{\alpha \beta}}.
\end{equation}
Inequality \eqref{ineq_PL} is optimal and equality holds if for some $0<\alpha \leq \beta$, $\overline x \in \R^N$ and $b>0$ we have
\begin{equation}\label{ineq PLb}
f(x)=-bL(x -\overline x),
\end{equation}
and
\begin{equation}\label{ineq PLc}
nh=\frac{\beta-\alpha (nh)^{-\frac{\beta-\alpha}{\alpha}}}{b\beta}.
\end{equation}
Moreover, we obtain the following ultracontractive bound
\begin{equation}\label{ultracontr}
\norm{e^{Q_n f}}_{\infty}\leq \norm{e^f}_1\tonde{\frac{1}{nh}}^{\frac{N}{2}}\tonde{2\pi}^{-\frac{N}{2}},
\end{equation}
and the equality holds if $nh=\frac{1}{b}$.
\end{thm}
\begin{proof}
In order to prove inequality \eqref{ineq_PL}, we are going to use the following Pr\'ekopa-Leindler inequality:\\
let $a, b >0$  such that $a+b=1$, and $u, v, w$ three non-negative functions on $\R^N$; suppose that  for any $x, y \in \R^N$ we have
\begin{equation}\label{hypPL}
u(x)^a v(y)^b\leq w(ax+by).
\end{equation}
Then the following inequality holds
$$\tonde{\int_{\R^N}u(x)dx}^{a}\tonde{\int_{\R^N}v(y)dy}^{b}\leq \int_{\R^N}w(x)dx.$$
Let $\alpha, \beta \in \R$ be such that $0<\alpha\leq \beta$ and set $\theta=\frac{\beta-\alpha}{\alpha nh} \beta$. For any $x \in \R^N$  we define
\begin{eqnarray*}
u(x)& = & e^{\beta Q_n f(x)},\\
v(x)& = & e^{-\theta   \frac{|x |^2}{2}},\\
w(x) & = & e^{\alpha f\tonde{\frac{\beta}{\alpha}x}}.
\end{eqnarray*}
We prove that $u$, $v$ and $w$ verify the hypothesis of the Pr\'ekopa-Leindler inequality \eqref{hypPL} with
\begin{equation}\label{Param}
 a=\frac{\alpha}{\beta}, \quad \quad b=\frac{\beta-\alpha}{\beta}.
\end{equation}
 By \eqref{SGh}  for any $\{q_1,\dots,q_n\}\subset \R^N$
\begin{equation*}
u(x)^a =e^{a \beta Q_n f(x)}\le e^{\alpha f(x-h\sum_{k=0}^{n-1} q_k)+\alpha h  \sum_{k=0}^{n-1} \frac{|q_k|^2}{2}},
\end{equation*}
hence
\begin{equation*}
 u(x)^a v(y)^b\leq e^{\alpha f(x-h\sum_{k=0}^{n-1} q_k)+\alpha h \sum_{k=0}^{n-1}  \frac{|q_k|^2}{2}-\theta \frac{\beta-\alpha}{2 \beta} |y|^2}.
\end{equation*}
If in particular we choose $q_0=q_1= \ldots=q_{n-1}=q$ in such a way that
\begin{equation*}
x-h\sum_{k=0}^{n-1}  q_k =x+\frac{\beta-\alpha}{\alpha}y
\end{equation*}
then we get
\begin{equation*}
 q =-\frac{\beta-\alpha}{\alpha nh}y
\end{equation*}
and therefore
\begin{align*}
 u(x)^a v(y)^b\leq e^{\alpha f(x-h\sum_{k=0}^{n-1} q_k)+\alpha \frac{h}{2}\sum_{k=0}^{n-1} \tonde{-\frac{\beta-\alpha}{\alpha nh}}^2 |y |^2-\frac{(\beta-\alpha)^2}{2 \alpha nh}\sum_{k=0}^{n-1} |y_k|^2}\\
=e^{\alpha f(x+\frac{\beta-\alpha}{\alpha}y)}=e^{\alpha f\quadre{\frac{\beta}{\alpha}(\frac{\alpha}{\beta}x+\frac{\beta-\alpha}{\beta}y)}}=e^{\alpha f(\frac{\beta}{\alpha}(ax+by))}=w(ax+by).
\end{align*}
Hence we  can apply the Pr\'ekopa-Leindler inequality to obtain
\begin{eqnarray*}
\tonde{\int e^{\beta Q_n f(x)} dx}^a \tonde{\int e^{-\theta   \frac{|x|^2}{2}}dx}^b\leq \int e^{\alpha f\tonde{\frac{\beta}{\alpha}x}} dx\\
\Rightarrow \tonde{\int e^{\beta Q_n f(x)} dx}^{\frac{1}{\beta}}\leq \tonde{\int e^{\alpha f\tonde{\frac{\beta}{\alpha}x}}dx}^{\frac{1}{\alpha}} \tonde{\int e^{-\theta    \frac{|x|^2}{2}}dx}^{-\frac{\beta-\alpha}{\alpha \beta}}.
\end{eqnarray*}
It follows that
\begin{equation}\label{ineq1}
\norm{e^{Q_n f}}_{L^{\beta}}\leq \tonde{\frac{\alpha}{\beta}}^{\frac{N}{\alpha}}\norm{e^f}_{ \alpha} \tonde{\int e^{-\theta   \frac{|x|^2}{2}}dx}^{-\frac{\beta-\alpha}{\alpha \beta}}.
\end{equation}
We compute
\begin{eqnarray*}
\tonde{\int e^{-\theta      \frac{|x|^2}{2}}dx}^{-\frac{\beta-\alpha}{\alpha \beta}} =  \tonde{\int e^{- \left|\tonde{\frac{\theta }{2}}^{\frac12}  x \right|^2} dx}^{-\frac{\beta-\alpha}{\alpha \beta}}
=\tonde{\int e^{-|y|^2} dy\tonde{\frac{\theta }{2}}^{-\frac{N}{2}}}^{-\frac{\beta-\alpha}{\alpha \beta}}
\\=\quadre{\tonde{\frac{\theta }{2}}^{-\frac{N}{2}} \pi^{\frac{N}{2}}}^{-\frac{\beta-\alpha}{\alpha \beta}}
=\tonde{\frac{2 \pi}{\theta }}^{-\frac{N}{2}\frac{\beta-\alpha}{\alpha \beta}}
=\quadre{\frac{2 \pi \alpha nh}{\beta (\beta-\alpha)}}^{-\frac{N}{2}\frac{\beta-\alpha}{\alpha \beta}}.
\end{eqnarray*}
Substituting in \eqref{ineq1}   we get \eqref{ineq_PL}.\\
In order to prove the optimality, we compute the terms $\norm{e^{Q_n f}}_{ \beta}$ and $\norm{e^f}_{ \alpha}$ appearing in \eqref{ineq_PL} for $f$ as in \eqref{ineq PLb}. We obtain
\begin{eqnarray*}
\norm{e^{Q_n f}}_{ \beta}  =  \tonde{\int e^{\beta Q_n f(x)} dx}^{\frac{1}{\beta}}=\tonde{\int e^{-\frac{b \beta}{1-nhb}L(x-\overline{x})} dx}^{\frac{1}{\beta}}\\
                                                     =  \tonde{\int e^{-\frac{b \beta}{1-nhb}H(z) }dz}^{\frac{1}{\beta}}
                                                     =  \tonde{\int e^{-\frac{b \beta}{1-nhb}\frac{z^2}{2}} dz}^{\frac{1}{\beta}}\\
                                                     =  \quadre{\frac{2 \pi (1-nhb)}{b \beta}}^{\frac{N}{2 \beta}}
\end{eqnarray*}
and
\begin{eqnarray*}
\norm{e^f}_{ \alpha}  =  \int \tonde{e^{-\alpha b L(x-\overline{x})} dx}^{\frac{1}{\alpha}}\\
                                             =  \int \tonde{e^{-\alpha b \frac{(x-\overline{x})^2}{2}} dx}^{\frac{1}{\alpha}}
                                             =  \int\tonde{e^{-\alpha b \frac{z^2}{2}} dz}^{\frac{1}{\alpha}}
                                             =  \tonde{\frac{2 \pi}{\alpha b}}^{\frac{N}{2 \alpha}}.
\end{eqnarray*}
and we obtain an equality in \eqref{ineq_PL} for
$$1-nhb=\frac{\alpha}{\beta}\tonde{\frac{1}{nh}}^{\frac{\beta-\alpha}{\alpha}},$$
i.e. \eqref{ineq PLc}.\\
The ultracontractive bound \eqref{ultracontr} is obtained for  $\beta\to+\infty$ and $\alpha=1$ in \eqref{ineq_PL}. Furthermore, if
$$\norm{e^{Q_n f}}_{ \infty}=1, \quad \norm{e^f}_{1}=\tonde{\frac{2\pi}{b}}^{\frac{N}{2}},$$
  the equality in \eqref{ultracontr} is obtained for $nh=\frac{1}{b}$.
\end{proof}
\begin{rem}
Consider the constant appearing in
\eqref{ineq_PL}, that is
\begin{equation}\label{constant}
C= \big( {\frac{\alpha}{\beta}}\big)^{\frac{N}{\alpha \beta}\frac{\alpha+\beta}{2}}\tonde{\frac{\beta-\alpha}{nh}}^{\frac{N}{2}\frac{\beta-\alpha}{\alpha \beta}}\tonde{2\pi}^{-\frac{N}{2}\frac{\beta-\alpha}{\alpha \beta}}.
\end{equation}
We observe that for fixed values of  $\alpha=\frac \beta 2, nh=1$, we have
$$C=
(\frac12)^{\frac{3N}{2\beta}}(\frac{\beta}{4\pi})^{\frac{N}{2\beta}}.$$
\begin{figure}[h!]\label{fig2}
\begin{center}
\includegraphics[scale=0.8]{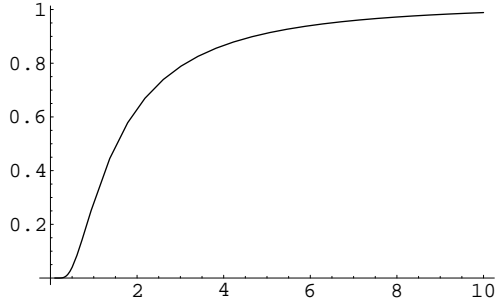} \quad \quad
\end{center}
\caption{Behaviour of the constant  \eqref{constant} as a function of $\beta$, fixed  $\alpha=\frac \beta 2$, $N=1$.}
\end{figure}

In this case the graph of the constant coincides with the one of the constant  in \eqref{ineq_gentil}  for $t=nh$ (see Fig. \ref{fig2}).
 \end{rem}
 In the next proposition we give a  hypercontractivity estimate with respect to the Lebesgue
 measure for the discrete semigroup $Q_n$  similar to one of Theorem \ref{stimeF}.
\begin{prop}
For any $f$ Lipschitz continuous and   for any $n$, any $h>0$, any increasing sequence $\{\beta_k\}_{k\in \N}$ with $\beta_0>0$, we have
\begin{equation}\label{const_Leb}
\norm{e^{Q_n f}}_{\beta_n}\leq \norm{e^f}_{ \beta_0}\prod_{k=1}^n \tonde{\frac{\beta_{k-1}}{\beta_k}}^{\frac{N}{\beta_{k-1}}}\tonde{\frac{2 \pi \beta_{k-1}}{\beta_k-\beta_{k-1}}\beta_k h}^{-\frac{\beta_k - \beta_{k-1}}{2 \beta_k \beta_{k-1}}}.
\end{equation}
\end{prop}
\begin{proof}
As in the proof of Theorem  \ref{Leb}, we set
$$u(x)=e^{\beta Q_n f(x)}, \quad v(x)=e^{-\theta\frac{|x|^2}{2}}, \quad w(x)=e^{\alpha Q_{n-1}f\tonde{\frac{\beta}{\alpha}x}}$$
and we prove that $u, v$ and $w$ verify the hypothesis of the Pr\'ekopa-Leindler inequality, with $a$ and $b$ as in \eqref{Param}.
By \eqref{HJhb}
$$u(x)^a v(y)^b\leq e^{\alpha Q_{n-1}f(x-hq )+h \frac{q^2}{2}-\frac{\theta}{2}\frac{y^2}{2}}.$$
Choosing $q =-\frac{\beta-\alpha}{\alpha h}y$ and $\theta=-\frac{\beta-\alpha}{\alpha h}\beta$, we obtain
$$u(x)^a v(y)^b\leq e^{\alpha Q_{n-1}f\tonde{x+\frac{\beta-\alpha}{\alpha}y}}=e^{\alpha Q_{n-1}f\tonde{\frac{\beta}{\alpha}(ax+by)}}=w(x).$$
Hence we can apply the  the Pr\'ekopa-Leindler inequality and arguing as for estimate \eqref{ineq1},  we obtain
\begin{equation}\label{recursive}
\begin{split}
 \norm{e^{Q_n f}}_{ \beta}\leq \tonde{\frac{\alpha}{\beta}}^{\frac{N}{\alpha}}\norm{e^{Q_{n-1}f}}_{\alpha}\tonde{\int e^{-\theta \frac{x^2}{2}}}^{-\frac{\beta-\alpha}{\alpha \beta}}=\\
 \tonde{\frac{\alpha}{\beta}}^{\frac{N}{\alpha}}\norm{e^{Q_{n-1}f}}_{\alpha}\quadre{\frac{2 \pi nh \alpha }{\beta  (\beta - \alpha)}}^{-\frac{N}{2}\frac{\beta -\alpha}{\alpha \beta}}
 \end{split}
\end{equation}
For $\beta=\beta_n, \alpha=\beta_{n-1}$ in \eqref{recursive}, we get
\[\norm{e^{Q_n f}}_{\beta}\leq \tonde{\frac{\beta_{n-1}}{\beta_n}}^{\frac{N}{\beta_{n-1}}}\norm{e^{Q_{n-1}f}}_{ {\beta_{n-1}}}\quadre{\frac{2 \pi nh \beta_{n-1}}{\beta_n (\beta_n-\beta_{n-1})}}^{-\frac{N}{2}\frac{\beta_n-\beta_{n-1}}{\beta_{n-1} \beta_n}} .\]
Iterating the previous argument for $n-1, n-2,\dots, 0$ we finally get the hypercontractivity   estimate \eqref{const_Leb} for $Q_n$.
\end{proof}
\begin{rem}
In particular, if we  set
$$\beta_k=\beta_0+\rho k h, \quad \beta_{k-1}=\beta_0+\rho(k-1)h,$$
in \eqref{const_Leb} we have
$$\lim_{h\to 0}\tonde{\frac{\beta_{k-1}}{\beta_k}}^{\frac{N}{\beta_{k-1}}}=\lim_{h\to 0}\tonde{\frac{\beta_0+\rho(k-1)h}{\beta_0+\rho k h}}^{\frac{N}{\beta_0+\rho(k-1)h}}=1$$
and
\begin{align*}
&\lim_{h\to 0}\tonde{\frac{2 \pi \beta_{k-1}}{\beta_k-\beta_{k-1}}\beta_k h}^{-\frac{\beta_k - \beta_{k-1}}{2 \beta_k \beta_{k-1}}}=\\
& \lim_{h\to 0} \tonde{\frac{2 \pi(\beta_0+\rho(k-1)h)}{\rho h}(\beta_0+\rho k h)h}^{-\frac{\rho h}{2(\beta_0+\rho h)(\beta_0+\rho(k-1)h)}}=1.
\end{align*}
Comparing  the graph in  Fig.\ref{fig3}   with the graph in  Fig.\ref{fig1} we see that the constant in \eqref{const_Leb} converges to $1$
by values lower  than 1, whereas  the constant in \eqref{hyperconc2} by values greater than $1$.
\end{rem}
%
\begin{figure}[h!]
\begin{center}
\includegraphics[scale=0.8]{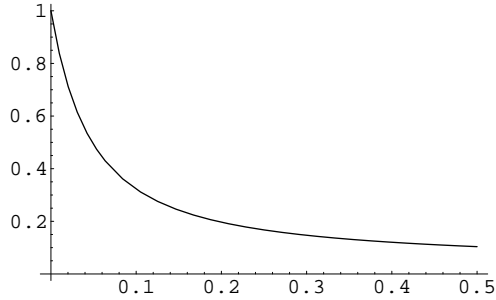} \quad \quad
\end{center}
\caption{Behavior of  $f(h)= \prod_{k=1}^{10} \tonde{\frac{\beta_{k-1}}{\beta_k}}^{\frac{N}{\beta_{k-1}}}\tonde{\frac{2 \pi \beta_{k-1}}{\beta_k-\beta_{k-1}}\beta_k h}^{-\frac{\beta_k - \beta_{k-1}}{2 \beta_k \beta_{k-1}}}$ for  $\rho=\beta_0=N=1$.}\label{fig3}
\end{figure}

\section{A concentration estimate  for the approximation error}\label{Section4}
It is well known  that for $h\to 0$, the discrete solution computed via the scheme \eqref{HJh}
converges uniformly to the solution of \eqref{HJ} with an error $\|u-u_h\|_\infty$   of order $h^{1/2}$ (see \cite{CL}, \cite{FF}).
In this section we obtain an   estimate of the measure of the set  where the error is concentrated.\par
To simplify the notation we write  $S_n f$ for $S_{nh}f$, where $S_t$ is the continuous semigroup associated to the equation \eqref{HJ}.
\begin{thm}\label{stimeErr}
If $\m$ is absolutely continuous with respect to the Lebesgue measure and satisfies the  logarithmic Sobolev inequality \eqref{LSI}, then
    for any $f$ semiconcave and  for any $n\in \N$,   $h>0$ and   $a \in\R$
    \begin{equation}\label{est1}
   \|e ^{ S_n f- Q_nf}\|_{\l_n}, \|e ^{Q_nf-S_n f}\|_{\l_n}\le  \prod_{k=1}^{n} (1+C\l_{k}^2h^2)^{\frac{1}{\l_{k}}}
    \end{equation}
     where $\l_n=a+\r n h$ and $\norm{\cdot}_p$ is the  $L^p$-norm associated to the   measure  $\mu$.
\end{thm}
\begin{proof}
Set
\begin{equation*}
 F_n=\|e ^{Q_n f-S_nf}\|_{\l_n}=\left(\int e^{\l_n (Q_nf(x)-S_nf(x))}d\m\right)^{\frac{1}{\l_n}}.
\end{equation*}
Arguing as in Theorem \ref{stimeF}, we arrive, see \eqref{h4}, to the inequality
\begin{equation*}
    \begin{split}
    & \l_nF_n^{\l_{n+1}}\left[\left(\frac{F_{n+1}}{F_n}\right)^{\l_{n+1}}-1\right]\le
    h\l_n^2\int e^{\l_n (Q_nf(x)-S_nf(x))}\frac{|DQ_n f-DS_nf|^2}{2}d\m\\
    &+\l_n\int e^{{\l_{n+1}(Q_{n}f-S_nf)}}
    \left[e^{{\l_{n+1}(Q_{n+1}f-Q_nf)-(S_{n+1}f-S_nf)}}-1\right]d\m+\\
   &\l_n \int e^{{\l_{n}(Q_nf(x)-S_nf(x))}} \left[e^{{h\r (Q_nf(x)-S_nf(x))}}-1-h\r (Q_nf(x)-S_nf(x))\right]d\m\\
    &-\l_nF_n^{\l_n}[e^{h\r\ln(F_n)}-1-h\r\ln(F_n)].
 \end{split}
\end{equation*}
Since $\l_n\le \l_{n+1}$,  $|e^{{h\r (Q_nf-S_nf)}}-1-h\r (Q_nf-S_nf)|\le Ch^2$ and  the last term  on the right hand side
of previous inequality is negative  we get
\begin{equation*}
\begin{split}
    & \l_nF_n^{\l_{n+1}}\left[\left(\frac{F_{n+1}}{F_n}\right)^{\l_{n+1}}-1\right]\le
    h\l_n^2\int e^{\l_n (Q_nf(x)-S_nf(x))}\frac{|DQ_n f-DS_nf|^2}{2}d\m\\
    &+\l_n\int e^{{\l_{n+1}(Q_{n}f-S_nf)}}
    \left[e^{{\l_{n+1}((Q_{n+1}f-Q_nf)-(S_{n+1}f-S_nf))}}-1\right]d\m+\l_nF_n^{\l_n}Ch^2=\\
    &\l_n\int e^{{\l_{n+1}(Q_{n}f-S_nf)}}
    \left[e^{{h\l_{n+1}(
    \frac{ Q_{n+1}f-Q_nf }{h}-\frac{ S_{n+1}f-S_nf }{h})}}-1-
    h\l_{n+1}\left(\frac{ Q_{n+1}f-Q_nf }{h}-\frac{ S_{n+1}f-S_nf }{h}\right)\right]d\m\\
    &+\l_n\int e^{{\l_{n+1}(Q_{n}f-S_nf)}}h\l_{n+1}\left(\frac{\partial S_n f}{\partial t}
    -\frac{ S_{n+1}f-S_nf }{h}\right)d\m+\\
    &\l_n\int e^{{\l_{n+1}(Q_{n}f-S_nf)}}
    h\l_n\left[\frac{ Q_{n+1}f-Q_nf }{h}-\frac{\partial S_nf}{\partial t}
    +\frac{|DQ_nf-DS_nf|^2}{2}\right]d\m+\\
    &h\l_n^2\int \big(e^{\l_n (Q_nf(x)-S_nf(x))}-e^{\l_{n+1}(Q_nf(x)-S_nf(x))}\big)\frac{|DQ_n f-DS_nf|^2}{2}d\m+
    \l_nF_n^{\l_n}Ch^2.
\end{split}
\end{equation*}
We have (see   \eqref{P3}  and the correspondent property for $S_t$)
\begin{equation*}
    \left|\frac{ S_{n+1}f-S_nf }{h}\right|,  \left|\frac{ Q_{n+1}f-Q_nf }{h}\right| \le C
\end{equation*}
and therefore
\begin{eqnarray}
    &e^{{h\l_{n+1}(
    \frac{ Q_{n+1}f-Q_nf }{h}-\frac{ S_{n+1}f-S_nf }{h})}}-1-
    h\l_{n+1}\left(\frac{ Q_{n+1}f-Q_nf }{h}-\frac{ S_{n+1}f-S_nf }{h}\right)\le C\l_{n+1}^2h^2\label{s2}\\[10pt]
    &e^{\l_n (Q_nf(x)-S_nf(x))}-e^{\l_{n+1}(Q_nf(x)-S_nf(x))}\le Ch\label{s3b}
\end{eqnarray}
and by Hopf-Lax formula
\begin{equation}\label{s3}
\begin{split}
     &\frac{\partial S_nf}{\partial t}(x)  -\frac{ S_{n+1}f(x)-S_nf(x) }{h}=- \sup_{q}\left\{q\cdot DS_{n}f(x)-L(q)
     \right\}\\
     &-\inf_{q}\left\{\frac{S_{n}f(x-hq)-S_nf(x)}{h}+L(q)\right\} \le C_2 h
    \end{split}
\end{equation}
where $C_2$ depends on the semiconcavity constant of $f$.
Moreover since $|P|^2/2=\sup_q\big\{q\cdot P-L(q)\big\}$
\begin{equation}\label{s4}
\begin{split}
    & \frac{ Q_{n+1}f-Q_nf }{h}+\frac{\partial S_nf}{\partial t}
    +\frac{|DQ_nf-DS_nf|^2}{2}\le\\
    &  -\sup_q\big\{-\frac{Q_nf(x-hq)-Q_nf(x)}{h}-L(q)\big\}+\frac{|DQ_nf|^2}{2}-\\
    &\left(\sup_q\big\{q\cdot DQ_nf-L(q)\big\}+\sup_q\big\{q\cdot DS_nf-L(q)\big\}\right)+\frac{|DQ_nf-DS_nf|^2}{2}\\
    &\le C_3h  -\sup_q\big\{q\cdot (DS_nf-DQ_nf)-L(q)\big\}+\frac{|DQ_nf-DS_nf|^2}{2}\le C_3h.
 \end{split}
\end{equation}
By \eqref{s2}, \eqref{s3}, \eqref{s3b} and \eqref{s4}, we get
\begin{equation*}
\begin{split}
    & \l_nF_n^{\l_{n+1}}\left[\left(\frac{F_{n+1}}{F_n}\right)^{\l_{n+1}}-1\right]\le
    \l_n C\l_{n+1}^2 h^2 F_n^{\l_n}
\end{split}
\end{equation*}
and therefore  we get
\begin{equation*}
   F_n^{\l_{n+1}}\le   F_n^{\l_{n+1}}(1+C\l_{n+1}^2h^2).
\end{equation*}
Iterating over $n$ and taking into account that $F_0=\|e ^{Q_0 f-S_0f}\|_{\l_n}=\|e ^{f-f}\|_{\l_n}=1$
we get the estimate
\begin{equation*}
    \|e ^{Q_nf-S_n f}\|_{\l_n}\le  \prod_{k=1}^{n} (1+C\l_{k}^2h^2)^{\frac{1}{\l_{k}}}.
\end{equation*}
Exchanging the role of $S_nf$ and $Q_nf$ we get the other estimate in \eqref{est1}.
\end{proof}
\begin{cor}\label{Cor1}
With the same notation of Theorem \ref{stimeErr}, if $f$ is semi-concave, then
for any $t\in [0,T]$, $t=nh$, we have
\begin{equation}\label{est2}
 \int(Q_nf-S_nf)d\m, \, \int(Q_nf-S_nf)d\m \le Ch
\end{equation}
with  $C$ depending on $T$ and the semi-concavity constant of $f$.
Moreover for any $p<1$
  \begin{equation}\label{est3}
  \m \{|S_nf-Q_nf|\ge h^p\}\le C e^{-1/h^{1-p}}.
\end{equation}
\end{cor}
\begin{proof}
We first observe that, since $e^t$ is a convex function, we have
$e^ {\int\l_n(Q_nf-S_n f)d\m}\le \int e^{\l_n(Q_nf-S_n f)}d\m$, hence by \eqref{est1}
\begin{align*}
e^ {\int\l_n(Q_nf-S_n f)d\m}\le \prod_{k=1}^{n} (1+C\l_{k}^2h^2) \le
\prod_{k=1}^{n} e^{C\l_{k}^2h^2}=e^{\sum_{k=1}^{n} C\l_{k}^2h^2}
\end{align*}
and therefore
\begin{equation}
\begin{split}
\int(Q_nf-S_n f)d\m\le \sum_{k=1}^{n} C\frac{\l_{k}^2}{\l_n}h^2\le C\sum_{k=1}^n\l_k  h^2=
C(na+\half n(n+1)h)h^2\le Ch
\end{split}
\end{equation}
for $t=nh\in [0,T]$ where $C$ depends on $T$ and semiconcavity constant of $f$.\\
To prove estimate \eqref{est3} observe that $\m \{|S_nf-Q_nf|\ge r\}=\m \{ S_nf-Q_nf \ge r\}+\m \{Q_nf-S_nf\ge r\}$ and
\begin{align*}
\m \{ S_nf-Q_nf \ge r\}\le
 \frac{1}{e^{\l_n r}}\int e^{\l_n(S_nf-Q_nf)}d\m \le e^{-\l_n r}\prod_{k=1}^{n} (1+C\l_{k}^2h^2)^{\l_n/\l_k}\le e^{-ar+Ch}.
\end{align*}
Taking $r=h^p$ and $a=\frac 1 h$ in the previous estimate we get \eqref{est3}
\end{proof}
\begin{rem}
The estimate \eqref{est3} can be interpreted as a concentration inequality of truncation error between the solution of the continuous
problem and of  the discrete one.
\end{rem}

\bibliographystyle{amsplain}

\end{document}